\documentclass[]{article}       
\usepackage{graphicx}
\usepackage{amsmath, amsfonts,amssymb,amsthm}        
\usepackage[a4paper,margin=2.5cm]{geometry}
\usepackage{mathtools}
\usepackage{url}
\usepackage{color}
\usepackage{enumitem}
\setlist{noitemsep,topsep=0pt}

\newtheorem{observation}{Observation}
\newtheorem{proposition}{Proposition}
\newtheorem{theorem}{Theorem}
\newtheorem{lemma}{Lemma}
\theoremstyle{definition}
\newtheorem{definition}{Definition}
\theoremstyle{remark}
\newtheorem{remark}{Remark}

\newcommand{\vvarepsilon}{\boldsymbol{\varepsilon}}
\newcommand{\vx}{\boldsymbol{x}}
\newcommand{\vy}{\boldsymbol{y}}
\newcommand{\valpha}{\boldsymbol{\alpha}}
\newcommand{\vbeta}{\boldsymbol{\beta}}
\newcommand{\vnull}{\boldsymbol{0}}
\newcommand{\vone}{\boldsymbol{1}}
\newcommand{\vell}{\boldsymbol{\ell}}
\newcommand{\vr}{\boldsymbol{r}}
\newcommand{\vs}{\boldsymbol{s}}
\newcommand{\vdelta}{\boldsymbol{\delta}}

\newcommand{\Sn}{\mathcal{S}_n}
\newcommand{\Pb}{\mathcal{P}_n(\vbeta)}
\newcommand{\T}{^{\mathrm{T}}}
\newcommand{\RR}{\mathbb{R}}

\begin{document}

\title{Walks on hyperplane arrangements and optimization of piecewise linear functions\thanks{This work was supported by Czech Science Foundation project 19-02773S.}%
}

    \author{Michal \v{C}ern\'{y}\thanks{Department of Econometrics, University of Economics, Prague, W.~Churchill Square~4, Prague~3, Czech Republic  
    ({\tt cernym@vse.cz}).}
    \and Milan Hlad\'ik\thanks{Department of Applied Mathematics, Faculty of Mathematics and Physics, Charles University, Malo\-stransk\'e n\'am. 25, Pragu 1e, Czech Republic ({\tt hladik@kam.mff.cuni.cz}).}
\and Miroslav Rada\thanks{Department of Econometrics and Department of Financial Accounting and Auditing, University of Economics, Prague, W.~Churchill Square~4, Prague~3, Czech Republic ({\tt miroslav.rada@vse.cz}).}}



\maketitle

\begin{abstract}
    We propose an exact iterative algorithm for minimization of a class of continuous cell-wise linear convex functions on a
    hyperplane arrangement. Our particular setup is motivated by evaluation of so-called rank estimators used in robust regression, where every cell of the underlying arrangement corresponds to a permutation of residuals (and we also show that the class of function for which the method works is more general). 
    The main obstacle in the construction of the algorithm is how to find an improving direction while standing in a point incident with exponentially many cells of the arrangement.
    We overcome this difficulty using Birkhoff Theorem which allows us to express the cone of improving directions in the exponential number of cells using a linear system with quadratic number of variables only.

\end{abstract}

\section{Introduction and motivation}

\subsection{Problem formulation}

Let $\Sn$ denote the set of all permutations of $\{1, \dots, n\}$.
Given $X \in \RR^{n \times p}$,
$\vy \in \RR^n$ and $\valpha \in \RR^n$, the task is to minimize the function
\begin{equation}
F(\vbeta) \coloneqq
\max_{\pi \in \Sn} \sum_{i=1}^n 
\alpha_{i}(y_{\pi(i)} - \vx_{\pi(i)}\T\vbeta),
\label{eq:defF}
\end{equation}
where $\vx_i\T$ stands for the $i$th row of $X$. 

Thorough the paper, we assume that 
\begin{equation}
\alpha_1 \leq \alpha_2 \leq \cdots \leq \alpha_n.
\label{eq:alpha}
\end{equation} 
It will be shown in Observation \ref{obs:ass:alpha:wlog} that assumption \eqref{eq:alpha} is without loss of generality.

Assumption \eqref{eq:alpha} is motivated by the statistical viewpoint. 
In Section~\ref{sect:da} we will summarize the significance 
of the function $F$ in robust regression. 
Our notation, which might be considered as unusual in optimization, follows the standards in statistics: $X$ is interpreted as a matrix of explanatory variables in a linear regression model, $\vy$ is interpreted as the dependent variable and the minimizer of $F$ then gives an estimate of regression parameters.
In this context, the numbers
$$
e_i^{\vbeta} \coloneqq y_i - \vx_i\T\vbeta,\quad i = 1, \dots, n
$$
are called \emph{residuals}.

In addition to the short summary in Section \ref{sect:da}, a reader can find more information on the statistical theory behind \eqref{eq:defF} e.g. in the book by Hettmansperger and 
McKean~\cite{hettmansperger:2010:RobustNonparametricStatistical}. In this paper we focus just on the optimization aspects of \eqref{eq:defF}.

\subsection{Basic properties}

Assumption \eqref{eq:alpha} says that $\alpha_1,\ldots, \alpha_n$ are ordered. Observation \ref{obs:ass:alpha:wlog} shows that regardless of the ordering the $\alpha$-coefficients, the infimum of $F(\beta)$ is the same. 

\begin{observation}
    \label{obs:ass:alpha:wlog}
    For every permutation $\sigma \in S_n$ the following equality holds true: $$\inf_{\vbeta} \max_{\pi \in \Sn} \sum_{i=1}^n \alpha_{\sigma(i)} e_{\pi(i)} = \inf_{\vbeta} \max_{\pi \in \Sn} \sum_{i=1}^n \alpha_i e_{\pi(i)}.$$
\end{observation}
\begin{proof} If $\sigma$ is an arbitrary permutation, then

\begin{equation}\tag*{\qed}
\max_{\pi \in \Sn} \sum_{i=1}^n \alpha_{\sigma(i)} e_{\pi(i)} 
=\max_{\pi \in \Sn} \sum_{i=1}^n \alpha_{i} e_{\pi(\sigma^{-1}(i))} 
=\max_{\pi' \in \Sn} \sum_{i=1}^n \alpha_i e_{\pi'(i)}.
\end{equation}
\end{proof}

From (\ref{eq:defF}) it follows that $F$ is a maximum of finitely many linear functions. Thus:

\begin{observation}\label{obs:cont:convex} 
    $F$ is continuous and convex. 
\end{observation}

The ordering of residuals will play an important role in the entire text.

\begin{definition}\hfill
    \label{def:consitent:permutation}
\begin{itemize}
\item[(a)] A permutation $\pi\in\Sn$ is called
\emph{consistent with $\vbeta$} if
\begin{equation}
e^{\vbeta}_{\pi(1)} \leq 
e^{\vbeta}_{\pi(2)} \leq \cdots \leq e^{\vbeta}_{\pi(n)}.
\label{eq:consistent:permutation}
\end{equation}
\item[(b)] The set of all permutations consistent with $\vbeta$
is denoted by $\Pb$.
\end{itemize}
\end{definition}

The next property shows that it is not necessary to compute the $|\Sn| = n!$ terms in (\ref{eq:defF}). Instead, the value $F(\vbeta)$ can be computed in $O(n\log n)$ time by sorting.

\begin{observation}
$F(\vbeta) = \sum_{i=1}^n \alpha_{i}e_{\pi^*(i)}^{\vbeta}$,
where $\pi^*$ is an arbitrary permutation
consistent with $\vbeta$. 
\label{obs:polytime}
\end{observation}

\begin{proof} 
    Let $\pi \in \Sn$ such that $e^{\vbeta}_{\pi(k)} > e^{\vbeta}_{\pi(\ell)}$ for some
$k < \ell$ and let $\pi' \in \Sn$ result from $\pi$ by a transposition interchanging $k$ and $\ell$. We will show that 
\begin{equation}
\sum_{i=1}^n \alpha_{i}e_{\pi'(i)}^{\vbeta}
\geq \sum_{i=1}^n \alpha_{i}e_{\pi(i)}^{\vbeta}.
\label{eq:qneq}
\end{equation}
Since $\pi^*$ can be obtained from any $\pi\in\Sn$ by a sequence of such transpositions, we get 
$\sum_{i=1}^n \alpha_{i}e_{\pi^*(i)}^{\vbeta} \geq
\sum_{i=1}^n \alpha_{i}e_{\pi(i)}^{\vbeta}$ for all $\pi\in\Sn$. Thus
\begin{align}
\sum_{i=1}^n \alpha_{i}e_{\pi^*(i)}^{\vbeta} = \max_{\pi\in\Sn} \sum_{i=1}^n \alpha_{i}e_{\pi(i)}^{\vbeta} = F(\vbeta).
\label{eq:maxx}
\end{align}

\emph{Proof of (\ref{eq:qneq})}. By assumption we have 
$\Delta e \coloneqq e_{\pi(k)}^{\vbeta} - e_{\pi(\ell)}^{\vbeta} > 0$.
Assumption~(\ref{eq:alpha}) implies that $\Delta\alpha \coloneqq \alpha_{\ell} - \alpha_{k} \geq 0$. Now
\begin{align*}
\sum_{i=1}^n \alpha_{i}e_{\pi'(i)}^{\vbeta}
&= \sum_{i \neq k, \ell} \alpha_{i}e_{\pi'(i)}^{\vbeta}
    + \alpha_{k}e_{\pi'(k)}^{\vbeta} + 
    \alpha_{\ell}e_{\pi'(\ell)}^{\vbeta}
\\ & =
\sum_{i \neq k, \ell} \alpha_{i}e_{\pi(i)}^{\vbeta}
    + \alpha_{k}e_{\pi(\ell)}^{\vbeta} + 
    \alpha_{\ell}e_{\pi(k)}^{\vbeta}
\\ & =
\sum_{i \neq k, \ell} \alpha_{i}e_{\pi(i)}^{\vbeta}
    + \alpha_{k}e_{\pi(\ell)}^{\vbeta} + 
    (\alpha_{k} + \Delta\alpha)(e_{\pi(\ell)}^{\vbeta} + \Delta e)
\\ & =
\sum_{i \neq k, \ell} \alpha_{i}e_{\pi(i)}^{\vbeta}
    + (\alpha_{k} + \Delta\alpha)e_{\pi(\ell)}^{\vbeta} + 
    \alpha_k(e_{\pi(\ell)}^{\vbeta} + \Delta e) + \Delta\alpha\Delta e
\\ & =
\sum_{i=1}^n \alpha_{i}e_{\pi(i)}^{\vbeta} + \Delta\alpha\Delta e
 \geq \sum_{i=1}^n \alpha_{i}e_{\pi(i)}^{\vbeta}.
 \tag*{\qedhere}
\end{align*}

\end{proof}

\begin{observation}
Minimization of $F(\beta)$ can be expressed as the linear programming problem
\begin{align}\label{eq:min:obs:LPform}
\min_{\substack{t\in\RR \\ \vbeta\in\RR^p}}\ t\ \ \text{s.t.}\ \ t \geq \sum_{i=1}^n \alpha_{i}
(y_{\pi(i)} - \vx_{\pi(i)}\T\vbeta)\quad \forall \pi\in\Sn.
\end{align}\label{obs:LPform}
\end{observation}

However, due to the huge number of constrains, standard linear programming (LP) algorithms do not yield a polynomial-time method.
This is what makes the problem nontrivial. 

\subsection{Complexity}

Propositions \ref{pro:p:completeness} and \ref{pro:p:complete:unboundedness} show that a general LP can be reduced to \eqref{eq:min:obs:LPform}: 
indeed, the decision version of \eqref{eq:min:obs:LPform} is $P$-complete. 
(We recommend \cite{greenlaw:1995:LimitsParallelComputation} as an excellent reference book on $P$-completeness theory.)

\begin{proposition}
    \label{pro:p:completeness}
It is P-complete to check whether the optimal value of (\ref{eq:defF}) is nonpositive.
\end{proposition}

\begin{proof}
Checking solvability of a system of linear inequalities $X\vbeta\geq \vy$ is known to be P-complete. Solvability of the system is equivalent to the claim that the optimal value of the linear program
\begin{align}\label{eq:min:obs:Pcompl}
\min_{\substack{t\in\RR \\ \vbeta\in\RR^p}}\ t \ \ \text{s.t.}\ \ X\vbeta+t\vone\geq \vy
\end{align}
is nonpositive, where $\vone=(1,\dots,1)\T$. Put $\alpha_1=\dots=\alpha_{n-1}=0$ and $\alpha_n=1$. Then \eqref{eq:min:obs:LPform} takes the form of \eqref{eq:min:obs:Pcompl}.\qed

\end{proof}

Checking unboundedness of $F(\vbeta)$ from below is P-complete, too.

\begin{proposition}
    \label{pro:p:complete:unboundedness}
It is P-complete to check whether $F(\vbeta)$ is unbounded from below.
\end{proposition}

\begin{proof}
We use the fact that checking strict solvability of a homogeneous system of linear inequalities $X\vbeta>0$ is known to be P-complete.     Solvability of the system is equivalent to solvability of $X\vbeta+t\vone\geq 0$, $t<0$. Put $\alpha_1=\dots=\alpha_{n-1}=0$ and $\alpha_n=1$. Since \eqref{eq:min:obs:LPform} is always feasible, we get that it is unbounded if and only if there are $t,\vbeta$ such that
$$
t<0,\ \ t \geq \sum_{i=1}^n \alpha_{i} (- \vx_{\pi(i)}\T\vbeta)\quad \forall \pi\in\Sn,
$$
which reduces to $X\vbeta+t\vone\geq 0$, $t<0$.
\end{proof}\qed

\subsection{Our contribution and related work}
\label{sect:justification}

\paragraph{Our WoA algorithm versus a polynomial algorithm based on the ellipsoid method.}
In this text we propose method for minimizing $F$, called
\emph{Walk-on-Arrangement (``WoA'') Algorithm}, 
which solves the problem using at most $O(n^{2p-2})$ auxiliary LPs.  
It follows that if 
\begin{equation}
p = O(1),\label{eq:poone}
\end{equation} 
then WoA yields a polynomial method. 

This is a weak result compared to the polynomiality theorem from \cite{cerny:2019:classoptimizationproblemsa}, where a (theoretically) efficient algorithm based on the ellipsoid method was designed.
The crucial question is: \emph{Does WoA make sense when we already have an unconditionally polynomial method?}  

The situation is similar to what has been well known in linear programming since 1979's Khachiyan's Theorem.
Does the simplex method with exponential worst-case complexity make sense once we have the ellipsoid method which works in polynomial time? The polynomial method for minimization of $F$ from
\cite{cerny:2019:classoptimizationproblemsa} is based on two procedures: the ellipsoid method 
for oracle-given polyhedra \cite{grotschel:1993:Geometricalgorithmscombinatorial} and Diophantine approximation \cite{schrijver:2000:TheoryLinearInteger}. Both of these procedures, although polynomial in theory, are extremely hard to implement for numerical reasons: the computation involves rational numbers with huge bit sizes. And the poor practical behavior of the ellipsoid method, well known from computational studies (see e.g.~\cite{gill:1980:NumericalInvestigationEllipsoid}), manifests itself in case of \cite{cerny:2019:classoptimizationproblemsa} as well. So it makes sense to have another method for minimization of $F$, possibly with superpolynomial worst-case bound, but with a potential to work better in many practical cases and with less serious numerical problems than those concerning the ellipsoid methods and Diophantine approximation. 

\paragraph{Gradient-descent methods}
Since $F$ is convex it is natural to try to minimize $F$ with gradient-based methods. 
If we are given a point $\vbeta$ and the function $F$ is smooth in~$\vbeta$, the situation is easy.
The main problem is what to do in a point $\vbeta$ where $F$ is not smooth. It is not obvious how to find a representative of the subgradient of $F$ 
in~$\vbeta$. One of the main contributions of this paper is a construction showing that the subgradient can be described by a linear system with polynomially bounded size.
In fact, the system is as small as $O(n^2)$. (As far as we are aware, this is an original---and possibly surprising---result; the key idea to achieve such a compact representation of the subgradient is found in Claim C in the proof of Lemma~\ref{lem:Lthree}.)

Comparison of WoA with gradient-descent methods is an important issue. We will return to it in Section~\ref{sect:gradient}, where we construct an example showing that 
the behavior of (a version of) gradient-descent search can be significantly poorer than the behavior of the WoA algorithm. Moreover, we will show that testing whether a given $\beta$ is a minimizer of $F$ requires essentially the same techniques as utilized by the WoA algorithm.

\subsection{Motivation from statistics}\label{sect:da}

We conclude the introductory section by a short summary from mathematical statistics showing the importance of $F(\vbeta)$ in robust regression. This is the main source of motivation and the main area of applications. For example, it could make sense to implement WoA as an R-package.

Consider the linear regression relationship
$$
\mathsf{E}[y|\vx] = \vx\T\vbeta^*,
$$
where $y$ is the dependent variable, $\vx \in \RR^p$ is the vector of explanatory variables and $\mathsf{E}[y|\vx]$ stands for the conditional expectation. We measure $n$ observations $(y_i, \vx_i)_{i=1,\dots,n}$ and the task is to estimate the unknown vector $\vbeta^*\in\RR^p$ of regression parameters. In matrix notation we have 
$\mathsf{E}[\vy|X] = X\vbeta^*$, where $\vy = (y_1, \dots, y_n)\T$ and $X$ is the matrix with rows $\vx\T_1, \dots, \vx\T_n$. There exist various estimators of $\vbeta^*$, such as least squares $\vbeta_{\text{LS}} = (X\T X)^{-1}X\T \vy$, which ``works well'' if $X$ has full column rank and the vector $\vvarepsilon = (\varepsilon_1, \dots, \varepsilon_n)\T$ 
of disturbances $\varepsilon_i \coloneqq y_i - \vx_i\T\vbeta^*$ satisfies additional assumptions, such independence, identical distribution and 
$\mathsf{E}[\vvarepsilon|X] = 0$.
Least squares is a prominent example of estimators based on \emph{minimization of loss functions}. Indeed, $\vbeta_{\text{LS}}$ minimizes the loss function $F(\vbeta) = \sum_{i=1}^n (y_i - \vx_i^T\vbeta)^2$ (known as \emph{residual sum of squares}). 

Rank estimators, or R-estimators for short, are based on minimization of different loss functions. 
These estimators have been designed as
robust estimators for the case when data $(y_i, \vx_i)_{i=1,\dots,n}$
are contaminated by outliers. (This is one of practically frequent cases when the assumptions necessary for the ``good behavior'' of $\vbeta_{\text{LS}}$ are violated.) The R-estimator $\vbeta_{\mathrm{R}}$ can be formulated as follows. Fix a \emph{score function} $\varphi(\xi)$ defined on $(0,1)$ satisfying: (i)~$\varphi$ is nondecreasing, (ii)~$\varphi(\frac{1}{2}+\xi) = -\varphi(\frac{1}{2}-\xi)$ for all $0 \leq \xi < \frac{1}{2}$, and (iii)~$\int_{0}^1 \varphi^2(\xi)\ \mathrm{d}\xi= 1$. If $n$ observations are available, set 
\begin{equation}
\alpha_i = \varphi\left(\frac{i}{n+1}\right), \quad i = 1, \dots, n.
\label{eq:score}
\end{equation}

Then, define a loss function $F(\vbeta)$ (which will be subsequently minimized) by this procedure: given a candidate 
estimate $\vbeta$, compute the residuals $e_i^{\vbeta} = y_i - \vx_i\T\vbeta$, 
sort them and output their weighted sum $F(\vbeta) = \sum_{i=1}^n w_i e_i^{\vbeta}$, where the weights $w_i$ \emph{depend on the ranks of the residuals in the sorted sequence}. (This justifies the name ``rank estimator''.) The smallest residual is assigned weight 
$\alpha_1$, the second smallest residual is assigned weight $\alpha_2$ and so on. For example, if we are given $\vbeta\in\RR^p$ and we find out that $e_5^{\vbeta} \leq e_1^{\vbeta} \leq e_4^{\vbeta} \leq e_3^{\vbeta} \leq e_2^{\vbeta}$, then $F(\vbeta) = \alpha_1 e_5^{\vbeta} + \alpha_2 e_1^{\vbeta} + \alpha_3 e_4^{\vbeta} + \alpha_4 e_3^{\vbeta} + \alpha_5 e_2^{\vbeta}$. This procedure defines a function $F(\vbeta)$. Its minimizer is the R-estimator $\vbeta_{\mathrm{R}}$.

Using more formal notation, we get the following formulation. Given $\vbeta \in \RR^p$, find $\pi^* \in \Pb$ 
and define 
$F(\vbeta) = \sum_{i=1}^n \alpha_{i} (y_{\pi^*(i)} - \vx_{\pi^*(i)}\T\vbeta)$.
This is the loss function to be minimized. It is easy to see that
$F(\vbeta)$ is well-defined---indeed, if $|\Pb|\geq 2$, then the procedure generates a value $F(\beta)$ independent of the choice of $\pi^*\in\Pb$. Not surprisingly, $F(\vbeta)$
is exactly what we saw in Observation~\ref{obs:polytime}, and thus
the loss function is exactly $F(\vbeta)$ as defined in (\ref{eq:defF}).

\begin{remark}[Examples of score functions.] Various score functions
are studied and used. Some of well-known examples include
the \emph{sign score function} 
$\varphi(\xi) = \text{sgn}(\xi-\frac{1}{2})$, 
\emph{Wilcoxon score function}
$\varphi(\xi) = \sqrt{12}(t-\frac{1}{2})$
and \emph{van der Waerden score function}
$\varphi(\xi) = \Phi^{-1}(t-\frac{1}{2})$, where
$\Phi^{-1}$ is the quantile function of $N(0,1)$.
\end{remark}

\begin{remark} In statistical literature, various assumptions
are needed to achieve good statistical properties of R-estimators (such as low finite-sample bias, consistency etc.). 
From the optimization viewpoint, such assumptions are not necessary.
For example, we do not need the usual assumption that $X$, the matrix of regressors, has full column rank. 
In this sense, our model is a bit more general than the one used in statistics.

\end{remark}

\begin{remark}[Complexity.]
    In applications of regression modeling (see e.g. \cite{fox:2008:AppliedRegressionAnalysis,wooldridge:2012:IntroductoryEconometricsModern} for basic references),
it is often happens that $n \gg p$. Except for some special situations, such as model selection problems, we usualy have a moderate number $p$ of regression parameters and a larger number $n$ of observations (in asymptotic theory one even takes $n \rightarrow \infty$). This shows that the assumption (\ref{eq:poone}) can be regarded as realistic, at least in some situations. Recall that (\ref{eq:poone}) implies that WoA is a polynomial-time algorithm as discussed in Section~\ref{sect:justification}. 
\end{remark}

\section{The Walk-on-Arrangement (WoA) method for minimization of $F(\vbeta)$}

The WoA algorithm will be presented in Section~\ref{sect:thealgo}. First, in Sections \ref{sect:geometry} and \ref{sect:permutation}, we sketch some geometric and algebraic notions utilized by the algorithm.

\begin{figure}[tb]
\centering
\includegraphics[width=\textwidth]{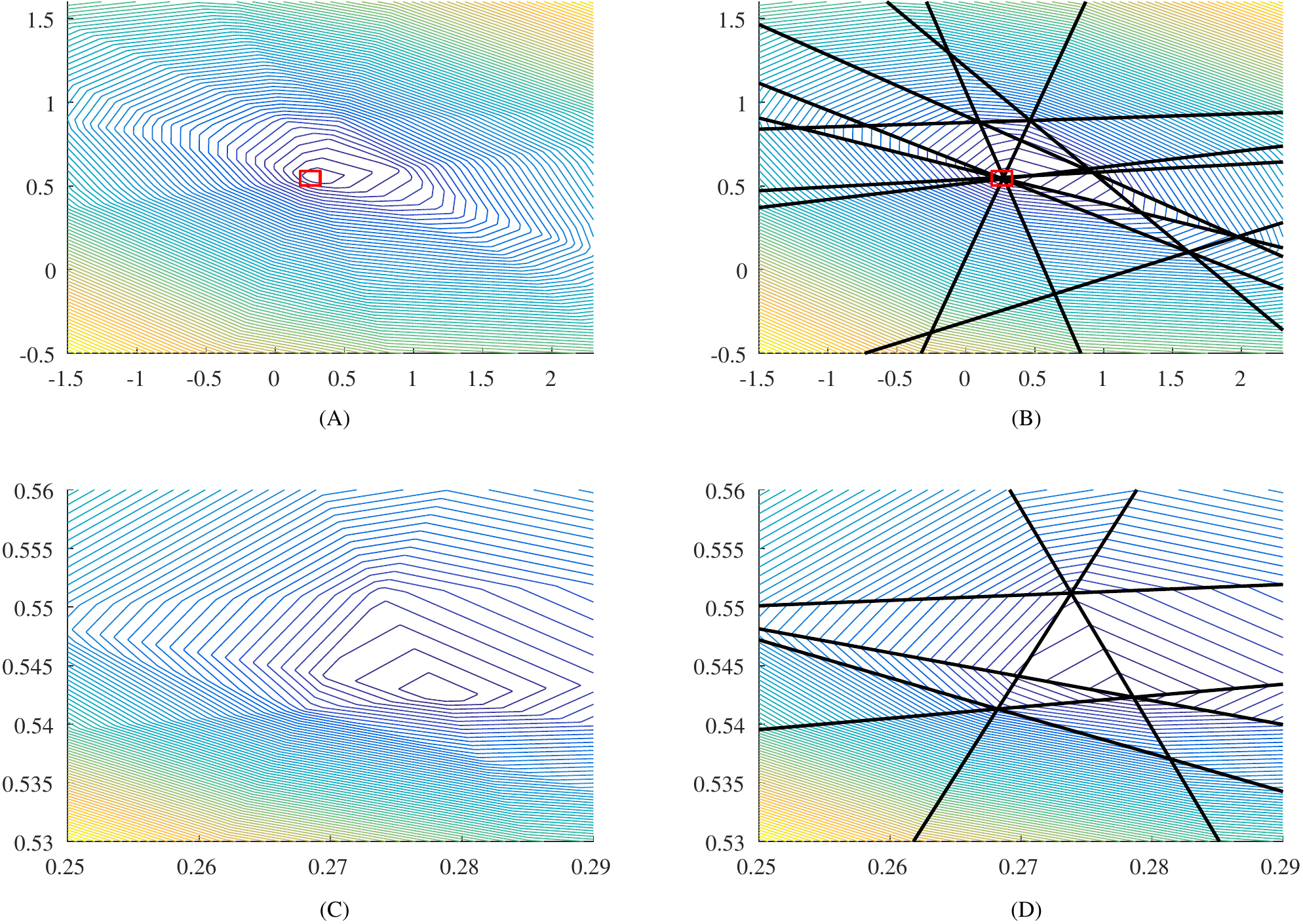}
\caption{(A):~Contour map of a sample functon $F(\vbeta)$ with $p = 2$ and $n = 5$. 
 (B):~The same contour map and the arrangement $\mathfrak{A}$ formed by $N=10$ hyperplanes 
$H_{ij}$, $1 \leq i < j \leq 5$. 
 (C, D): Zoom of the neighbourhood of the minimizer in the red box from (A, B). (In this example, the minimizer is unique and is attaned in a point $\vbeta$ satisfying $e^{\vbeta}_{i_1} = e^{\vbeta}_{i_2} = e^{\vbeta}_{i_3}$ for some $i_1 < i_2 < i_3$. In general, the polyhedron of minimizers can have any dimension.)} 
\label{fig:one}
\end{figure}

\subsection{Geometry of $F(\vbeta)$}\label{sect:geometry}

There is a partition of $\RR^p$ into a finite number of 
polyhedral regions such that $F(\vbeta)$ is linear on each of the regions. Let us formalize this geometric insight more precisely. 

For $1 \leq i < j \leq n$ define 
$$
H_{ij} \coloneqq \{\vbeta\in\RR^p\ |\ y_i - \vx_i\T\vbeta = y_j - \vx_j\T\vbeta\}. 
$$
The system $\{H_{ij}\}_{1 \leq i < j \leq n}$
defines a hyperplane arrangement $\mathfrak{A}$ (see Figure~\ref{fig:one} for an example). The arrangement
can be seen as a partition of $\RR^p$ into polyhedral regions---the connected regions of $\RR^p \setminus \bigcup_{i,j} H_{ij}$, or (by convention) their closures.

Given a permutation $\pi \in \Sn$, define
$$
C^{\pi} \coloneqq \{\vbeta\in\RR^p\ |\ 
y_{\pi(1)} - \vx_{\pi(1)}\T\vbeta \leq
y_{\pi(2)} - \vx_{\pi(2)}\T\vbeta \leq
\cdots
\leq 
y_{\pi(n)} - \vx_{\pi(n)}\T\vbeta\}.
$$
Obviously, the set $C^{\pi}$ is a convex polyhedron. 

\begin{definition}\label{def:cell}
If $C^{\pi} \neq \emptyset$, it is called \emph{a cell} (of $\mathfrak{A}$ corresponding to the permutation~$\pi$).
\end{definition}

A cell can be bounded or unbounded. It can have full dimension or may be dimension-deficient (this can happen e.g.~if there are $(i,j) \neq (i',j')$ such that $H_{ij} = H_{i',j'}$).

\begin{remark}
The notion ``cell of an arrangement'' may vary in literature, cf.~\cite{edelsbrunner:1986:Constructingarrangementslines} for example. Here we used a definition making the forthcoming theory as simple as possible, with no need to handle possible degenerate cases as exceptions. To avoid misunderstanding: a closure of a connected region of
$\RR^p \setminus \bigcup_{i,j} H_{ij}$ is always a cell in the sense of Definition~\ref{def:cell}, but the opposite implication need not hold true. 
\end{remark}

Obviously, the system of cells covers the entire space $\RR^p$. The system of full-dimensional cells can be regarded as another representation of $\mathfrak{A}$.

The system of cells plays an important role because $F$ is a cell-wise linear function (meaning that restriction of the domain of $F$ onto a cell produces a linear function). 

We also get the natural correspondence
$$
\text{$\pi \in \Pb$ 
\ \ \ $\Longleftrightarrow$ \ \ \ $\vbeta \in C^{\pi}$.}
$$

Typically, for many permutations $\pi\in\Sn$ we have $C^{\pi} = \emptyset$, as shown by the following lemma.

\begin{lemma}[a corollary of \cite{buck:1943:Partitionspace} and \cite{zaslavsky:1975:FacingarrangementsFacecount}]
An arrangement in $\RR^p$ with $N$ hyperplanes has at most $2\sum_{i=0}^{p-1}\binom{N-1}{i} = O(N^{p-1})$ cells.
\end{lemma}

Now, $\mathfrak{A}$
has $N = \binom{n}{2} = O(n^2)$
hyperplanes, and thus 
\begin{equation}
\text{the number of cells of $\mathfrak{A}$} = O(n^{2p-2}). 
\label{eq:noc}
\end{equation}

\subsection{Permutations}
\label{sect:permutation}

Recall that a \emph{permutation matrix} $P_{\pi}$ corresponding to a permutation $\pi \in \Sn$ is defined as 
$$
(P_{\pi})_{ij} = \left\{
\begin{array}{ll}
1 & \text{if $\pi(i) = j$}, \\
0 & \text{otherwise}.
\end{array}
\right.
$$ 

\begin{definition}
\begin{itemize}
\item[(a)] An index pair $(i,j) \in \{1, \dots, n\}^2$ is called
\emph{$\vbeta$-active} if $\exists \pi \in \Pb$ such that $\pi(i) = j$. Otherwise
the pair is \emph{$\vbeta$-inactive}.
\item[(b)] The set of $\vbeta$-active index pairs is denoted by $A(\vbeta)$
and the set of $\vbeta$-inactive index pairs is denoted by $\overline{A}(\vbeta)$.
\end{itemize}
\end{definition}

A $\vbeta$-active index pair has the following interpretation:
for a $\pi\in\Pb$, the corresponding permutation matrix 
$P_{\pi}$ can have ones only in positions $(i,j) \in A(\vbeta)$.
More formally we can say that 
$\sum_{\pi\in\Pb} (P_{\pi})_{ij} = 0$ iff $(i,j) \in \overline{A}(\vbeta)$.

Recall also that a square matrix is \emph{bistochastic}
(or \emph{doubly stochastic}) if all elements are nonnegative and all columns and rows sum up to one.
The following property---Birkhoff's Theorem~\cite{birkhoff:1946:Threeobservationslinear}---will be useful:
\emph{every bistochastic matrix results as a convex combination of some permutation matrices}. More precisely: 

\begin{lemma}[Birkhoff's Theorem \cite{birkhoff:1946:Threeobservationslinear}] If $G$ is a bistochastic matrix, then there exists a number $M \geq 1$, a system of permutations 
$\pi_1, \dots, \pi_{M} \in \Sn$ and coefficients $\lambda_1, \dots, \lambda_M > 0$ satisfying 
$\sum_{m=1}^M \lambda_m = 1$ and $\sum_{m=1}^M \lambda_m P_{\pi_m} = G$.
\label{lem:birkhoff}
\end{lemma} 

\begin{remark}
It is interesting that $M$ can be always taken as small as $O(n^2)$. (We do not need this bound in the forthcoming theory. We simply couldn't resist the temptation to mention this beautiful fact.)
\end{remark}

\subsection{The algorithm}\label{sect:thealgo}

In this section we describe the WoA algorithm formally. The proof of correctness and in-depth discussion is the merit of Section \ref{sect:corrc}. In Section 2.4 we illustrate the geometry behind the algorithm.

\vspace{.2cm}

\leftskip0cm\rightskip0cm
\noindent
\textsc{{Walk-on-Arrangement (WoA) Algorithm.}} 

\textbf{Input:} Data $X \in \RR^{n \times p},
\vy \in \RR^{n}, \valpha\in\RR^{n}$ s.t.~$\alpha_1 \leq
\alpha_2 \leq \cdots \leq \alpha_n$ and an initial 
point $\vbeta \in \RR^p$.
 
\textbf{Step~1~(Sorting).}
Sort the residuals $e_1^{\vbeta}, \dots, e_n^{\vbeta}$
and
find (any) permutation $\pi \in \Pb$.

\textbf{Step~2~(Minimization over the cell $C^\pi$).}
Solve the linear programming problem
\begin{equation}
\min_{\widetilde\vbeta\in\RR^p} \sum_{i=1}^n \alpha_{i}(y_{\pi(i)} - \vx_{\pi(i)}\T\widetilde\vbeta)
\quad \text{s.t.}\quad \widetilde\vbeta \in C^{\pi}. 
\label{eq:LPI}
\end{equation}

\textbf{Step~3~(Unboundedness test I).}
If (\ref{eq:LPI}) is unbounded, then \texttt{stop}, function $F$ is unbounded from below. 

\textbf{Step~4~(Move to the argmin of (\ref{eq:LPI})).}
Let $\vbeta^*$ be (any) optimal solution of (\ref{eq:LPI}).
Construct the set $A(\vbeta^*)$ of $\vbeta^*$-active index pairs. 

\textbf{Step~5~(Auxiliary linear system).} 
Use linear programming to solve the system 
\begin{subequations}
    \label{eq:auxLP}
\begin{align}
\alpha_i \vx_{j}\T\vell +  s_i + r_j &\geq \phantom{-}0\quad\quad \forall (i,j) \in A(\vbeta^*),
\label{eq:auxLPone}
\\
\sum_{i=1}^n s_i + \sum_{j=1}^n r_j &= -1 
\label{eq:auxLPtwo}
\end{align}
\end{subequations}
with variables $\vell \in \RR^p, \vr \in \RR^n, \vs \in \RR^n$.

\textbf{Step~6~(Optimality test).}
If the system (\ref{eq:auxLPone}b) is infeasible, 
then \texttt{stop}, $\vbeta^*$ is a minimizer.
Otherwise, let $(\vell^*, \vr^*, \vs^*)$ be (any) solution
of (\ref{eq:auxLPone}b).

\textbf{Step~7~(Using $\vell^*$ as an improving direction).}  
For all $1 \leq i < j \leq n$ compute 
$$
d_{ij} \coloneqq \left \{
\begin{array}{ll}
-\infty & \text{if $\vx_j\T\vell^* = \vx_i\T\vell^*$,} \\
\frac{y_j - \vx_j\T\vbeta^* - (y_i - \vx_i\T\vbeta^*)}{\vx_j\T\vell^* - \vx_i\T\vell^*}
 & \text{otherwise} 
\end{array}\right.
$$
and set $D \coloneqq \{d_{ij} \ |\ d_{ij} > 0\}$.
Observe that $d_{ij}$ has the following property:
\begin{equation}
    \text{if }d_{ij} \neq -\infty,\text{ then }\vbeta^* + d_{ij}\vell^* \in H_{ij}.
    \label{eq:property:d:ij}
\end{equation}
   
\textbf{Step~8~(Unboundedness test II).} If $D = \emptyset$ then \texttt{stop},
function $F$ is unbounded from below.

\textbf{Step~9~(Line search in direction $\vell^*$ and update of $\vbeta$).} 
Let $d^* \in \text{argmin} 
\{F(\vbeta^* + d\vell^*)\ |\ d \in D\}$
and update $\vbeta \coloneqq \vbeta^*+ d^*\vell^*$.

\textbf{Step~10.}~\texttt{Goto} Step 1. 

\leftskip0cm\rightskip0cm

\subsection{An illustration}\label{sect:illu}

\begin{figure}[tb]
\centering
\includegraphics[scale=1.2]{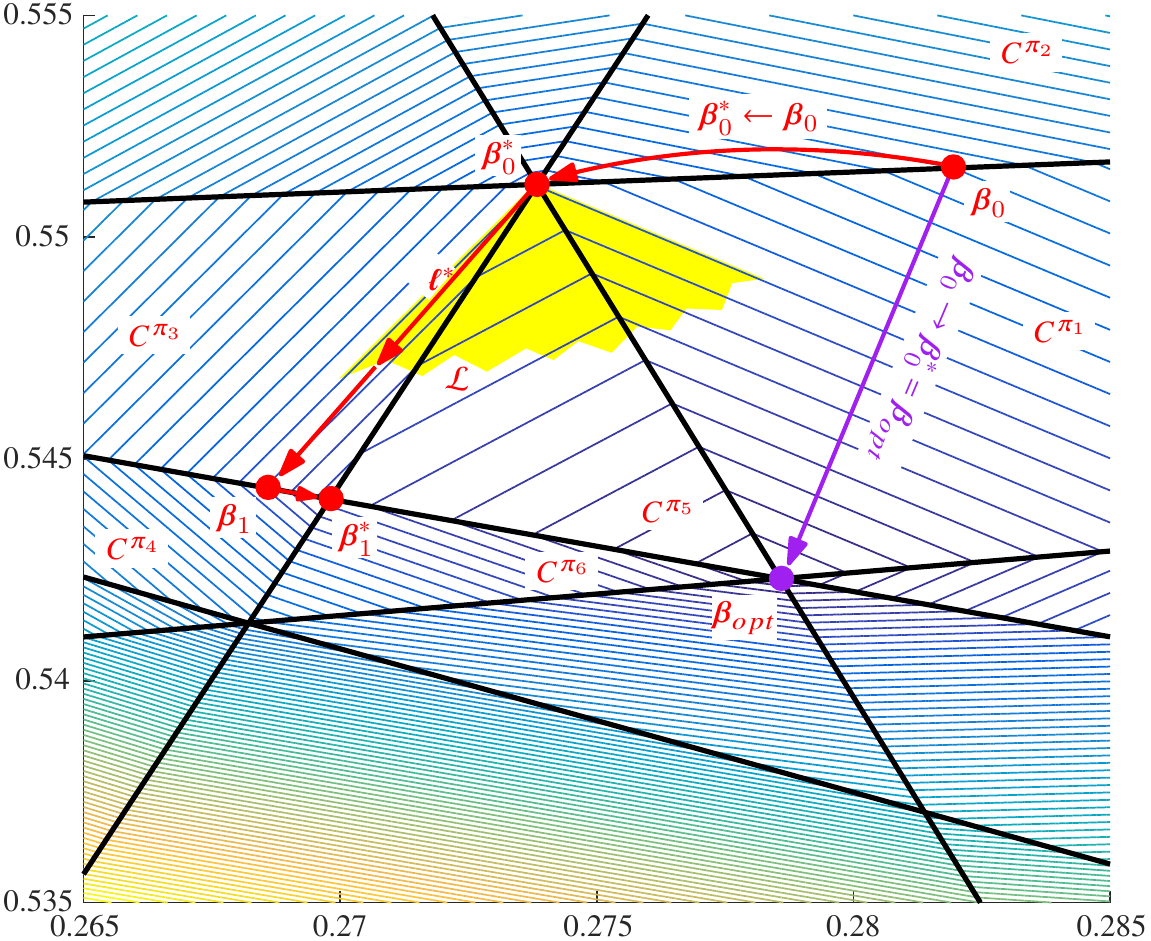}
\caption{An illustration of WoA (for details see Section~\ref{sect:illu}).} 
\label{fig:woa:iterations}
\end{figure}

Figure~\ref{fig:woa:iterations} shows the problem from Figure~\ref{fig:one}(D), where WoA is started from $\vbeta_0$. Here, $\mathcal{P}_n(\vbeta_0) = \{\pi_1,\pi_2\}$,
and thus selection of $\pi \in \mathcal{P}_n(\vbeta_0)$ in Step~1 affects whether the algorithm will continue in cell $C^{\pi_1}$ or in cell $C^{\pi_2}$. 
(So, a particular procedure selecting $\pi$ in Step~1 if multiple exist can be called ``pivoting strategy''.)

\begin{itemize}
\item
If $\pi_1$ is selected, then Step 2 finds $\vbeta^*_0$ (depicted by the purple arrow) and Step~6 determines that a minimizer 
$\vbeta_{opt} = \vbeta^*_0$ has been found.

\item
If $\pi_2$ is selected, then the process goes to $C^{\pi_2}$. Step 2 finds $\vbeta^*_0$ (depicted by the red arrow) and Step~6 determines that $\vbeta^*_0$ is not a minimizer. The yellow cone $\mathcal{L}$ is the cone of improving directions. One of such directions, $\vell^*$, is found in Step~5. Then, Step~9 performs line search in direction $\vell^*$ and finds $\vbeta_1 = \vbeta_0^* + d^*\vell^*$. 

In the next iteration we have $\mathcal{P}_n(\vbeta_1) = \{\pi_3, \pi_4\}$. However, in this case, no matter which one is selected in Step~1, Step~2 will determine the same point $\vbeta^*_1$. 

The procedure then continues in a similar manner (not shown in the figure): an improving direction $\vell^*$ can lead from $\vbeta^*_1$ either to $C^{\pi_5}$ or $C^{\pi_6}$. It is apparent that in both cases, the minimizer will be found in Step~2 and confirmed in Step~6 of the following iteration.  
\end{itemize}

\begin{remark}
We did not specify a particular ``pivoting strategy'' for selection of $\pi\in\Pb$ in Step~1. And, similarly,
we did not give guidance for the choice of the improving direction $\vell^*$ in Step 5 if multiple exist (this can be also seen as a kind of pivoting). Nevertheless, 
the correctness theorem (Theorem~\ref{theo:correctness}, Section~\ref{sect:corrc}) implies that the method converges with \emph{arbitrary} pivoting strategies.
\end{remark}

\subsection{Correctness theorem}\label{sect:corrc}

This section is devoted to the proof of the main result.

\begin{theorem}[correctness]
After at most $O(n^{2p-2})$ iterations, WoA algorithm correctly finds a minimizer of $F$ or states that $F$ is unbounded from below.\label{theo:correctness}
\end{theorem}
 
Before we turn our attention to the proof, it will be useful to prove three lemmas. The crucial proof ``trick''---showing how to encode a possibly large number of permutations into a bistochastic matrix---is present in Claim~C inside the proof of Lemma~\ref{lem:Lthree}.

\begin{lemma}\label{lem:Lone}
Let $\vbeta \in \RR^p$.
Exactly one of the systems
(\ref{eq:Done}b)
and
(\ref{eq:Pone}--e)
is feasible: 
\begin{subequations}
\begin{align}
\alpha_i \vx_{j}\T\vell +  s_i + r_j &\geq 0\quad \forall (i,j) \in A(\vbeta), \label{eq:Done}
\\
\sum_{i=1}^n s_i + \sum_{j=1}^n r_j & < 0, \label{eq:Dtwo}
\end{align}
\end{subequations}
where $\vell \in \RR^p$, $\vr \in \RR^n$, $\vs \in \RR^n$ are variables; and
\begin{subequations}
\begin{align}\hspace{2cm}
\sum_{j=1}^n \vx_{j} \sum_{i=1}^n \alpha_i G_{ij} &= \vnull,
&& \hspace{2cm} \label{eq:Pone} \\
\sum_{i=1}^n G_{ij} &= 1 && \forall j = 1, \dots, n, \hspace{2cm} \label{eq:Ptwo}\\
\sum_{j=1}^n G_{ij} &= 1 &&\forall i = 1, \dots, n, \label{eq:Pthree}\\
G_{ij} &\ge 0, && \forall (i,j) \in A(\vbeta), \label{eq:Pfour}\\
G_{ij} &= 0, && \forall (i,j) \in \overline{A}(\vbeta),  \label{eq:Pfive}
\end{align}
\end{subequations}
where $G_{ij}$ ($i,j = 1, \dots, n$) are variables.
\end{lemma}

\begin{proof} The statement follows directly from Farkas' Lemma since
the systems form a primal-dual pair (if 
the redundant variables $G_{ij}$ with $(i,j) \in \overline{A}(\vbeta)$
are removed). 
\end{proof}

\begin{lemma}\label{lem:Ltwo}
If $\vbeta \in \RR^p$ and $(\vr^*, \vs^*, \vell^*)$
is a solution of (\ref{eq:Done}b), 
then for any sufficiently small $d>0$ it holds true that $F(\vbeta)>F(\vbeta+d\vell^*)$.
\end{lemma}

\begin{proof}
Let $(\vell^*, \vr^*, \vs^*)$ be a solution
of (\ref{eq:Done}b). If $d > 0$ is sufficiently
small, then there exists 
\begin{align}\label{eq:piq}
\pi \in \Pb \cap \mathcal{P}_n(\vbeta + d\vell^*).
\end{align} 
In other words, if $d$ is susfficiently small, $\vbeta$ and $\vbeta + d\vell^*$ are in the same cell $C^\pi$. 

From (\ref{eq:Done}) select the inequalities
with $(i,j)$ satisfying $\pi(i) = j$ and sum them up:  
\begin{align*}
0 \leq \sum_{i = 1}^n 
\left(
\alpha_i  \vx_{\pi(i)}\vell^* +  s_i + r_{\pi(i)}\right) 
&=
\sum_{i=1}^n \alpha_i \vx_{\pi(i)}\T\vell^* 
+
\sum_{i=1}^n s_i
+
\sum_{j=1}^n r_{j}.
 \end{align*}
Since $\sum_{i=1}^n s_i + \sum_{j=1}^n r_{j} < 0$ 
by (\ref{eq:Dtwo}), we have proved
\begin{equation}
\sum_{i=1}^n \alpha_i \vx_{\pi(i)}\T\vell^* > 0.
\label{eq:krok}
\end{equation}
Now, using (\ref{eq:piq}) and (\ref{eq:krok}) we can write
\begin{align*}
F(\vbeta + d\vell^*) - F(\vbeta)
& =
\sum_{i=1}^n \alpha_i(y_{\pi(i)} - \vx_{\pi(i)}\T(\vbeta + d\vell^*))
-
\sum_{i=1}^n \alpha_i(y_{\pi(i)} - \vx_{\pi(i)}\T\vbeta)
\\ & =
-d \sum_{i=1}^n \alpha_i \vx_{\pi(i)}\T\vell^* < 0.
 \end{align*}
\end{proof}

\begin{lemma}\label{lem:Lthree}
Let $\vbeta\in\RR^p$ and let there exist a bistochastic $(n\times n)$-matrix $G$ satisfying 
\begin{itemize}
\item[(i)] $\sum_{j=1}^n \vx_{j} \sum_{i=1}^n \alpha_i G_{ij} = \vnull$
 and
\item[(ii)] $G_{ij} = 0$ for all $(i,j)\in\overline{A}(\vbeta)$.
\end{itemize}
Then $\vbeta$ is a minimizer of $F$.
\end{lemma}

\begin{proof}   We divide the proof into three claims. Let $\mathfrak{L}(\beta)$ be a shorthand for the linear programming problem
\begin{multline*}
\max_{\substack{\gamma_{\widetilde\pi} \in \RR:\\ \widetilde\pi \in \Pb}}
\Bigg\{\sum_{\pi\in\Pb} \!\!\!\! \gamma_{\pi} \sum_{i=1}^n \alpha_i y_{\pi(i)}\ 
\Bigg|
\sum_{\pi\in\Pb} \!\!\!\! \gamma_\pi\sum_{i=1}^n \alpha_i \vx_{\pi(i)} 
= \vnull,
\\
\sum_{\pi\in \Pb} \!\!\!\! \gamma_{\pi} = 1,
\ \ \gamma_\pi \geq 0\ \ \forall \pi\in\Pb
\Bigg\}
\end{multline*}
and let $L^*(\vbeta)$ stand for its optimal objective value.

\emph{\bf Claim A.} \emph{If $\mathfrak{L}(\vbeta)$ is feasible, then $F(\vbeta) = L^*(\vbeta)$.}

\emph{Proof of Claim A.} 
Observe that $F(\vbeta) = \sum_{i=1}^n \alpha_i (y_{\pi(i)}
- \vx_{\pi(i)}\T\vbeta)$ for \emph{any} permutation $\pi \in \Pb$. Thus,
if $\gamma_{\pi} \geq 0$ and $\sum_{\pi\in\Pb} \gamma_\pi = 1$, then
$$
\sum_{\pi\in\Pb} \!\!\!\! \gamma_{\pi} \sum_{i=1}^n \alpha_i (y_{\pi(i)}
- \vx_{\pi(i)}\T\vbeta) = \sum_{\pi\in\Pb} \!\!\!\! \gamma_{\pi} F(\vbeta) = F(\vbeta).
$$
In addition we have
\begin{align*}
\sum_{\pi\in\Pb} \!\!\!\! \gamma_{\pi} \sum_{i=1}^n \alpha_i (y_{\pi(i)}
- \vx_{\pi(i)}\T\vbeta) 
&= 
\sum_{\pi\in\Pb} \!\!\!\! \gamma_{\pi} \sum_{i=1}^n \alpha_i y_{\pi(i)}
- \sum_{\pi\in\Pb} \!\!\!\! \gamma_{\pi} \sum_{i=1}^n \alpha_i \vx_{\pi(i)}\T\vbeta 
\\ &
=
\sum_{\pi\in\Pb} \!\!\!\! \gamma_{\pi} \sum_{i=1}^n \alpha_i y_{\pi(i)}
 \end{align*}
since $\sum_{\pi\in\Pb} \gamma_\pi\sum_{i=1}^n \alpha_i \vx_{\pi(i)} 
= \vnull$ by the feasibility assumption.  \emph{End of proof of Claim A.}

\emph{\bf Claim B.} \emph{If $\mathfrak{L}(\vbeta)$ is feasible,
then $\vbeta$ is a minimizer of $F$.}

\emph{Proof of Claim B.} First observe that the objective function 
of $\mathfrak{L}(\beta)$ is bounded. So, under the assumption of feasibility, we can use the LP duality theorem. With Claim~A we get
\begin{align}
F(\vbeta) & = L^*(\vbeta) \nonumber
\\          & =
\text{the optimal objective value of the dual LP of $\mathfrak{L}(\vbeta)$}
                          \nonumber
 \\         & = \min_{\substack{t\in\RR, \\ \widetilde\vbeta\in\RR^p}}
\left\{t\ \Bigg|\
t \geq \sum_{i=1}^n \alpha_i(y_{\pi(i)} - \vx_{\pi(i)}\T\widetilde\vbeta)
\ \ \forall \pi \in \Pb
\right\} \label{eq:dualxx}
\\ 
& = \min\left\{ F(\widetilde\vbeta)\ \Bigg|\ 
\widetilde\vbeta \in \!\!\!\!\bigcup_{\pi\in\Pb} \!\!\!\!C^{\pi}\right\}
\nonumber\\ 
& = \min\{ F(\widetilde\vbeta)\ |\ {\widetilde\vbeta \in \RR^p}\}.
\nonumber
\end{align}
It it straightforward to verify that  (\ref{eq:dualxx}) is the dual of $\mathfrak{L}(\vbeta)$.
Expression (\ref{eq:dualxx}) resembles
Observation~\ref{obs:LPform} with the exception that we quantify over $\Pb$ rather than $\Sn$. In other words, expression (\ref{eq:dualxx}) tells us that $\beta$ is a \emph{local} minimizer (on a certain neighborhood of $\vbeta$ intersecting only cells $C^{\pi} \ni \vbeta$).
But, as $F$ is a convex function, a local minimizer must be global.
\emph{End of proof of Claim~B.}

\textbf{Claim C} (the key part of the proof)\textbf{.} 
\emph{Under the assumptions of Lemma~\ref{lem:Lthree},
$\mathfrak{L}(\vbeta)$ is feasible.}

\emph{Proof of Claim C.} By Birkhoff's Theorem (Lemma~\ref{lem:birkhoff}) there exists 
a number $M \geq 1$, a family $P_{\pi_1}, \dots, P_{\pi_M}$ of permutation matrices
(representing certain permutations $\pi_1, \dots, \pi_M \in \Sn$)
and coefficients $\lambda_1, \dots, \lambda_M > 0$ 
such that $\sum_{m=1}^M \lambda_m = 1$ and
$$
G = \sum_{m=1}^M \lambda_m P_{\pi_m}.
$$
Since all matrices $P_{\pi_m}$ are 0-1 matrices and
the coefficients $\lambda_m$ are positive, assumption 
(ii) implies the following: \emph{for every $m$, $(P_{\pi_m})_{ij} = 0$ for all $(i,j) \in \overline{A}(\vbeta)$.} In other words,
for every $m$ we have $\pi_m \in \Pb$. Now, for every $\pi \in \Pb$, define
\begin{equation}
\gamma_{\pi} \coloneqq \left\{
\begin{array}{cl}
\lambda_m & \text{if $\pi = \pi_m$ for some $m \in \{1, \dots, M\}$}, \\
0               &  \text{otherwise}.
\end{array} 
\right. 
\label{eq:gl}
\end{equation}
To summarize: we have $\gamma_{\pi} > 0$ for (some)
permutations $\pi \in \Pb$ and $\gamma_{\pi} = 0$
for all permutations $\pi \in \Sn \setminus \Pb$.

Let us make an easy but important observation: for any $(i,j)$ it holds true that
$$
\sum_{\substack{\pi\in\Pb:\\ \pi(i) = j}} \gamma_{\pi}
= \sum_{m=1}^M \lambda_m (P_{\pi_m})_{ij}.
$$
[Indeed, $(P_{\pi_m})_{ij} = 0$ if $\pi_m(i) \neq j$
and $(P_{\pi_m})_{ij} = 1$ if $\pi_m(i) = j$. So, both sides just sum up those $\lambda$'s corresponding to permutations
$\pi$ satisfying $\pi(i) = j$. On the left-hand side we can see $\gamma$'s, but by (\ref{eq:gl}) they are just $\lambda$'s renamed.]

Using assumption (i) we get
\begin{align*}
\vnull &= \sum_{j=1}^n \vx_j \sum_{i=1}^n \alpha_i G_{ij}
= \sum_{i=1}^n \alpha_i \sum_{j=1}^n \vx_j  G_{ij}
= \sum_{i=1}^n \alpha_i \sum_{j=1}^n \vx_j  
\sum_{m=1}^M \lambda_m (P_{\pi_m})_{ij}
\\ & =
\sum_{i=1}^n \alpha_i \sum_{j=1}^n \vx_j  
\sum_{\substack{\pi\in\Pb:\\ \pi(i) = j}} \gamma_{\pi}
=
\sum_{i=1}^n \alpha_i \sum_{j=1}^n   
\sum_{\substack{\pi\in\Pb:\\ \pi(i) = j}} \!\! \vx_j \gamma_{\pi}
\\&=
\sum_{i=1}^n \alpha_i   \!\! 
\sum_{\pi\in\Pb} \vx_{\pi(i)} \gamma_{\pi}
= \sum_{\pi\in\Pb} \!\!\gamma_{\pi} \sum_{i=1}^n \alpha_i \vx_{\pi(i)}.
\end{align*}
The last expression, together with $1 = \sum_{i=1}^M \lambda_m = 
\sum_{\pi\in\Pb} \gamma_{\pi}$ and $\gamma_\pi \geq 0$,
imply the feasibility of $\mathfrak{L}(\vbeta)$. \emph{End of proof of Claim C.}
\end{proof}

\noindent
\emph{Proof of Theorem~\ref{theo:correctness}.}
\emph{Case I: System (\ref{eq:auxLPone}b)
is feasible.}
If (\ref{eq:auxLPone}b)
is feasible, 
then $(\vs^*, \vr^*, \vell^*)$ solves (\ref{eq:Done}b). By Lemma~\ref{lem:Ltwo},
$\vell^*$ is an improving direction. The ``line-search'' function
$$g(\delta) \coloneqq F(\vbeta^* + \delta\vell^*)\ \ \text{defined for\ \ $\delta > 0$}
$$ 
is continuous, piecewise
linear and convex (these properties are inherited from~$F$)
and its line segments are connected in points $\delta = d_{ij} > 0$, where $d_{ij}$ are the numbers computed in Step~7. 
Note that $\beta^* + d_{ij}\ell^* \in H_{ij}$ by \eqref{eq:property:d:ij}. 

From Step~2 it is obvious that $F(\vbeta^*) \leq F(\vbeta)$. 
If $D \neq \emptyset$, the choice of $d^*$ implies that
$F(\vbeta^* + d^*\vell^*) < F(\vbeta^*)$. 

If $D = \emptyset$, then $g(\delta)$ is a linear function with negative slope; thus $g(\delta) = F(\vbeta^* + \delta\vell^*) \rightarrow -\infty$ for $\delta \rightarrow \infty$. 
(In other words, the ray 
$\{\vbeta^* + \delta\vell^*\ |\ \delta > 0\}$ lies in an unbounded cell $C^{\pi}$ and $F$, being a linear function on $C^{\pi}$, decreases along this ray.) 
This proves correctness of Step~8.

\emph{Case II: System (\ref{eq:auxLPone}b)
is infeasible.} In this case, also system (\ref{eq:Done}b) is infeasible.
[\emph{Proof.} If (\ref{eq:Done}b) has a solution $(\vell, \vr, \vs)$ satisfying $\sum_{i=1}^n s_i + \sum_{j=1} r_j = -\eta$ with $\eta > 0$, then
$(\frac{\vell}{\eta}, \frac{\vr}{\eta}, \frac{\vs}{\eta})$ is a solution of (\ref{eq:auxLPone}b).]

 By Lemma~\ref{lem:Lone}, system (\ref{eq:Pone}--e) is feasible. The matrix $G \equiv (G_{ij})_{i,j=1,\dots n}$ from (\ref{eq:Pone}--e) satisfies the assumptions of
Lemma \ref{lem:Lthree}. The Lemma guarantees that $\vbeta^*$ is a minimizer.

The unboundedness test in Step~3 is correct obviously. 

To summarize, we have proved:
\begin{itemize}
\item[(i)] The algorithm produces a sequence of points $\vbeta_1, \vbeta_2, \dots$ such that $F(\vbeta_1) > F(\vbeta_2) > \cdots$.
\item[(ii)] If the algorithm terminates, it outputs a correct answer (either ``an optimal point found'' or ``$F$ is unbounded'').
\end{itemize}

It remains to prove that the algorithm terminates. We show that \emph{no cell of $\mathfrak{A}$ is visited twice}. 
Let $\pi_k$ be the permutation found in Step 1 in the $k$th iteration of the algorithm. Let $\vbeta^*_k$ be minimum of \eqref{eq:LPI} computed in Step~2. If we find out in Steps~5 and 6 that $\vbeta^*_k$ is not the optimal point, we move to a point $\vbeta'_k$ in Step 9. Clearly, $F(\vbeta'_k) < F(\vbeta^*_k)$. 
Because $F(\vbeta^*_k)$ is the minimal value of $F(\vbeta)$ over $C^{\pi_k}$, it follows that $F(\vbeta_{k'}) \not\in C^{\pi_k}$ for all $k' \geq k+1$. Thus, the number of iterations is bounded by the number of cells, which is $O(n^{2p-2})$ by~(\ref{eq:noc}).
\qed

\section{Concluding remarks and comments}

\subsection{Strengthening the theory}

Various properties required for Theorem \ref{theo:correctness} are stated as implications. 
Nevertheless, it is worth noting that the converse implications hold true as well in most cases (even if the converse implications are not used in the proof):
\begin{itemize}
    \item Lemma \ref{lem:Ltwo} provides the crucial building block of WoA algorithm. It allows for finding a representative of the subgradient in Step~5. It says that any solution of system (\ref{eq:Done}b) can be used as the improving direction. The converse holds true, too: The system (\ref{eq:Done}b) for a given $\vbeta$ is feasible if and only if $\vbeta$ is not optimal. 
        
    \item Claims A, B and C in the proof of Lemma \ref{lem:Lthree} hold as equivalences, too. For Claims~A and B, the converse implications are straightforward. For Claim C, the converse implication reads ``If $\mathfrak{L}(\vbeta)$ is feasible, then properties (i)--(ii) in Lemma~\ref{lem:Lthree} hold true''. To fulfill the consequent of the implication, it is sufficient to consider the matrix $G$ in the form $G_{ij} = \sum_{\pi\in P(\vbeta):\pi(i)=j} \gamma_{\pi}$. As a result, Lemma~\ref{lem:Lthree} can itself be stated as an equivalence. 
\end{itemize}

\subsection{Comparison of WoA with generic gradient-descent (``GGD'') methods}\label{sect:gradient}%

Here we understand GGD in the following sense. Assume that an initial point $\vbeta$ is given. Then, repeat the following steps:

\vspace{.1cm}

    \begin{enumerate}[label=(\roman*)]
        \item If $\vbeta$ is a minimizer (or a point close to the minimizer), then \texttt{stop}.
\item If $F$ is smooth in $\vbeta$, take the gradient $\nabla F(\vbeta)$, perform a step in the direction $-\nabla F(\vbeta)$ (where the step length can be determined either by line search or by another strategy), and update $\vbeta$ accordingly.
\item If $F$ is not smooth in $\vbeta$, then calculate a ``small'' perturbation vector $\vdelta$ such that $F$ is smooth in $\vbeta + \vdelta$ and replace $\vbeta$ by $\vbeta+\vdelta$. [For example, one can consider a random perturbation $\vdelta$. 
Another example is a deterministic construction of $\vdelta$ guaranteeing that $F$ smooth in $\vbeta + \vdelta$; such $\vdelta$ can be constructed  using bit-size (``Big-$L$'') arguments.]
\end{enumerate}

\vspace{.1cm}
In the proof of Theorem~\ref{theo:correctness} we showed that WoA has the property that \emph{no cell of $\mathfrak{A}$ is visited twice}. This is the main advantage of WoA
compared to GGD. 
Indeed, the bad case which can happen to GGD is illustrated in Figure~\ref{fig:ggd:example}. The figure consists of two parts:
\begin{itemize}
    \item In the left part, GDD is run from a point $\vbeta_0$. In each iteration, say $i$th, length $d_i$ of a step in direction $\vell \coloneqq -\nabla F(\vbeta_i)$ is chosen via line search such that $F(\vbeta_i - d\vell)$ is minimized. This means that every step ends in a non-smooth point, denoted by $\vbeta_i^*$. As a perturbation strategy, the step in direction $\vell$ is prolonged slightly by a perturbation $\vdelta_i$.

        This results in a sequence $\vbeta_1, \vbeta_2, \ldots$, which repeatedly visits cells $C^{\pi_1}$ and $C^{\pi_2}$; the sequence approaches vertex $\vbeta'_0$. 

    \item In the right part, we assume that the algorithm is standing in vertex $\vbeta'_0$. WoA would find an improving direction in the yellow cone. Any perturbation ending in the union of cells $C^{\pi_5}$ and $C^{\pi_6}$, e.g. randomly chosen $\vdelta'_0$, is sufficient for GGD to continue. Even if we are able to find such a perturbation (which might be hard without WoA in general), GDD improves $F(\vbeta)$ only slowly, since the contours in $C^{\pi_5}$ and $C^{\pi_6}$ (and also the adjacent cells $C^{\pi_4}$ and $C^{\pi_5}$) are almost parallel.
\end{itemize}

\begin{figure}[tb]
\centering
\includegraphics[scale=1]{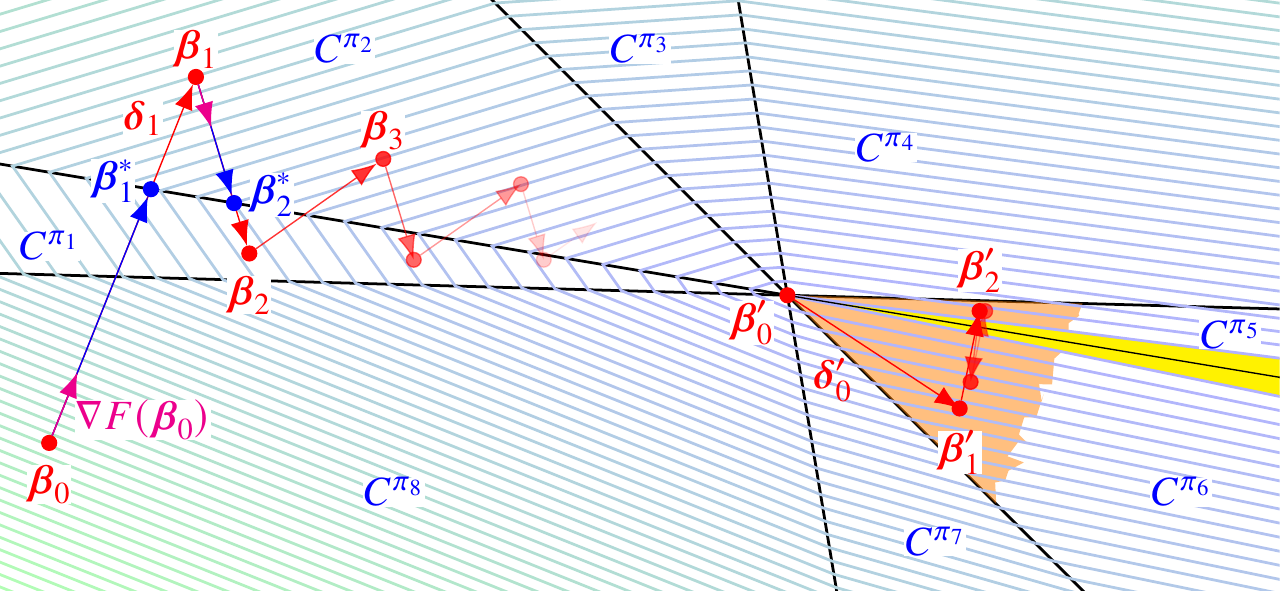}
\caption{An illustration of gradient-descent method (for details see Section \ref{sect:gradient}).} 
\label{fig:ggd:example}
\end{figure}

\begin{remark} 
Recall that the most difficult step is finding a representative of the subgradient of $F$ in a point $\vbeta$ where $F$ is not smooth. This problem can be either resolved by LP as in Step~5 of WoA (and then GGD would be essentially the same as WoA), or by a ``rescue'' strategy overcoming the problem. This is exactly step (iii) of GGD: in a non-smooth point, the perturbation will push $\vbeta$ to a smooth point from where GGD can continue.  
\end{remark}

\subsection{Practical considerations: Two-phase algorithms} 

The WoA algorithm can be initialized from an arbitrary point $\vbeta_0 \in \RR^p$ and it constructs a sequence of points $\vbeta_1, \vbeta_2, \dots, \vbeta_K$ such that $F(\vbeta_0) > F(\vbeta_1) > \cdots > F(\vbeta_K) = \min_{\vbeta\in\RR^p} F(\vbeta)$ 
or reports that $F$ is unbounded from below. As stated in Section~\ref{sect:justification},
we do not have a better bound than $K = O(n^{2p-2})$. In every iteration, one or two LPs are supposed to be solved. It seems to be reasonable to use two-phase implementations:
\begin{itemize}\item \emph{Phase I: Heuristic.} Do your best to get close to a minimum of $F$ heuristically. There are no limits for fantasy what one can try. For example, gradient-descent with some perturbation strategy (e.g. the one used in Section \ref{sect:gradient}) to escape non-smooth points can work well.
    
    When the heuristic method is convinced that a minimizer can be ``close'' to the best point 
$\widetilde\vbeta$ found so far, switch to Phase II. 
\item \emph{Phase II: Exact minimization.} 
Run WoA from $\widetilde\vbeta$, let it find a minimizer exactly (or determine that $F$ is unbounded) and hope that few iterations will suffice. 
\end{itemize}

Although we are joking about this strategy---emphasizing that its claimed performance is not supported by theory---it resembles, in a sense, the crossover strategy from LP solvers.
As far as we are aware, the most successful implementations follow
a similar two-phase scheme: first, long-step interior point methods (IPMs) are used to get close to an optimal point, and then the solver switches to a version of the simplex method which finds an optimal point exactly.
(The switch is sometimes refereed to as ``crossover step''. This strategy overcomes the tedious rounding step of IPMs which is slow and numerically sensitive.)

\subsection{Complexity and pivoting strategies}
The most challenging question regarding complexity of WoA algorithm is whether the iteration bound $O(n^{2p-2})$ from Theorem~\ref{theo:correctness} can be reduced and, namely, whether it can be reduced to $q(n,p)$, where $q$ is a polynomial. (This would imply unconditional polynomiality of WoA.) 
Or do there exist Klee-Minty-like instances showing the opposite?

Nevertheless, even if the polynomial iteration bound cannot be derived, one can still try to reduce the number of iterations at least heuristically. There are some degrees of freedom in the formulation of the algorithm. Namely:
\begin{itemize}
    \item Step 1 of each iteration of WoA starts with some $\vbeta$ and any permutation $\pi$ consistent with this $\vbeta$ is selected to work with. It is tempting to formulate a strategy how to choose ``a good'' consistent permutation, such that the sequence of visited point $\vbeta^*_0, \vbeta^*_1, \ldots$ will make $F(\vbeta)$ decrease quickly.
    \item A solution $(\vell, \vs, \vr)$ of linear system \eqref{eq:auxLP} is computed (if exists; otherwise the algorithm ends) in Step 5 of each iteration of WoA. Vector $\vell$ is then used an improving direction from the current point $\vbeta^*$. There are two natural strategies what $\vell$ shall be found:
        \begin{enumerate}
            \item Find an $\vell$ such that $\{\vbeta \mid \vbeta^* + d \vell,\ d \ge 0\}$ covers an edge of $\mathfrak{A}$. Then WoA traverses edges of cells of $\mathfrak{A}$. This type of improving directions resembles the simplex method and allows for a conceptual simplification of the algorithm. However, we currently do not know any efficient way to compute this type of improving directions.
            \item Find $\vell$ such that it is the steepest improving direction. 
 It can be found using the program 
                \begin{equation}
\begin{aligned}
    \min_{\vell \in \RR^p, \vr \in \RR^n, \vs \in \RR^n} \sum_{i=1}^n s_i &+ \sum_{j=1}^n r_j \\
    \text{subject to }\quad\alpha_i \vx_{j}\T\vell +  s_i + r_j &\geq \phantom{-}0\quad\quad \forall (i,j) \in A(\vbeta^*),\\
\lVert \vell \rVert &\le 1.
\end{aligned}
\label{eq:best:improving:direction}\end{equation}
Since \eqref{eq:best:improving:direction} is a second-order cone program, it is easily solvable with interior point methods.
        \end{enumerate}
\end{itemize}
Are there other promising strategies for selecting a promising consistent permutation and/or computing good improving direction?


%
%



\end{document}